\documentclass[11pt,reqno]{amsart}

\usepackage{amssymb}

\usepackage{amsfonts}

\usepackage{amsmath}

\usepackage[all]{xy}

\usepackage{color}

\newtheorem{theorem}{Theorem}[section]

\newtheorem{corollary}[theorem]{Corollary}

\newtheorem{lemma}[theorem]{Lemma}

\newtheorem{proposition}[theorem]{Proposition}

\newtheorem{Definition}[theorem]{Definition}

\newtheorem{Example}[theorem]{Example}

\newtheorem{Remark}[theorem]{Remark}

\newenvironment{remark}{\begin{Remark}\begin{em}}{\end{em}\end{Remark}}

\DeclareMathOperator{\tr}{tr}

\address{Jinmi Hwang \\ Department of Mathematics, Chungbuk National University, Cheongju 28644, Korea}
\email{jinmi0401@chungbuk.ac.kr}

\address{Sejong Kim \\ Department of Mathematics, Chungbuk National University, Cheongju 28644, Korea}
\email{skim@chungbuk.ac.kr}

\begin{document}

\title[Tensor product and Hadamard product for the Wasserstein means]{Tensor product and Hadamard product for the Wasserstein means}

\author{Jinmi Hwang and Sejong Kim}

\date{}
\maketitle

\begin{abstract}

As one of the least squares mean, we consider the Wasserstein mean of positive definite Hermitian matrices. We verify in this paper the inequalities of the Wasserstein mean related with a strictly positive and unital linear map, the identity of the Wasserstein mean for tensor product, and some inequalities of the Wasserstein mean for Hadamard product.

\vspace{5mm}


\end{abstract}

\section{Introduction and preliminaries}

It is a long-standing problem to define a barycenter (or a mean) of a finite number of points in a metric space. Given a probability vector $\omega = (w_{1}, \dots, w_{n})$, a natural and canonical barycenter is the least squares mean, which is a minimizer of the weighted sum of squares of distances to each point. In the open convex cone $\mathbb{P}_{m}$ of positive definite matrices, which we will consider throughout the paper, there are several different and important barycenters depending on the given distances. For instance, the arithmetic mean is the least squares mean in the real vector space $\mathbb{P}_{m}$ equipped with the Euclidean distance $d_{E} (A, B) = \Vert A - B \Vert_{2}$, and the Cartan mean is the least squares mean in the non-positive curvature space (CAT(0) space or Hadamard space) $\mathbb{P}_{m}$ equipped with the Riemannian trace distance $d_{R} (A, B) = \Vert \log A^{-1/2} B A^{-1/2} \Vert_{2}$. On the other hand, it is difficult to see whether such a minimizer exists, and whether the minimizer is unique if it exists. Recently a new metric, called the Wasserstein metric, and the least squares mean on our setting $\mathbb{P}_{m}$ have been introduced \cite{ABCM, BJL}.

For given $A, B \in \mathbb{P}_{m},$ the \emph{Wasserstein metric} $d(A,B)$ is given by
\begin{displaymath}
d(A,B)=\left[\tr\left(\frac{A+B}{2}\right)-\tr(A^{1/2}BA^{1/2})^{1/2}\right]^{1/2}.
\end{displaymath}
In quantum information theory, the Wasserstein metric is known as the Bures distance of density matrices. The unique geodesic connecting from $A$ to $B$ is given by
\begin{displaymath}
A \diamond_{t} B = (1-t)^{2}A + t^{2}B + t(1-t)\left[(AB)^{1/2}+(BA)^{1/2}\right], \ t \in [0,1].
\end{displaymath}
As the least squares mean for the Wasserstein metric, the \emph{Wasserstein mean} denoted by $\Omega(\omega; \mathbb{A})$ for $\mathbb{A}=(A_{1}, \cdots, A_{n})$ is defined by
\begin{equation}
\Omega(\omega; \mathbb{A}) = \underset{X \in \mathbb{P}_{m}} {\arg \min} \sum^{n}_{j=1} w_{j} d^{2}(X, A_{j}),
\end{equation}
and it coincides with the unique solution $X \in \mathbb{P}_{m}$ of the matrix nonlinear equation
\begin{equation} \label{E:Wass eq}
I = \sum_{j=1}^{n} w_{j} (A_{j} \# X^{-1}),
\end{equation}
where $A \# B = A^{1/2} (A^{-1/2} B A^{-1/2})^{1/2} A^{1/2}$ is the geometric mean of $A$ and $B$ in $\mathbb{P}_{m}$. From the equivalent equation \eqref{E:Wass eq} of the Wasserstein mean, many properties of the Wasserstein mean can be derived from the those of two-variable geometric mean, so we here list some of its properties: for any $A, B, C, D \in \mathbb{P}_{m}$
\begin{itemize}
% \item[(1)] $A \# B = A^{1/2} B^{1/2}$ if $A$ and $B$ commute.
\item[(G1)] $(a A) \# (b B) = \sqrt{a b} (A \# B)$ for any $a, b > 0$.
\item[(G2)] $A \# B = B \# A$.
\item[(G3)] $A \# B \leq C \# D$ whenever $A \leq C$ and $B \leq D$.
% \item[(5)] The map $\mathbb{P}_{m} \times \mathbb{P}_{m} \to \mathbb{P}_{m}, \ (A, B) \mapsto A \# B$ is continuous.
\item[(G4)] $X (A \# B) X^{*} = (X A X^{*}) \# (X B X^{*})$ for any nonsingular matrix $X$.
\item[(G5)] $(A \# B)^{-1} = A^{-1} \# B^{-1}$.
% \item[(8)] $[(1 - \lambda) A + \lambda B] \# [(1 - \lambda) C + \lambda D] \geq (1 - \lambda) (A \# C) + \lambda (B \# D)$ for any $\lambda \in [0,1]$.
\item[(G6)] $\det (A \# B) = \sqrt{\det A \det B}$.
\item[(G7)] $\displaystyle \left[ \frac{A^{-1} + B^{-1}}{2} \right]^{-1} \leq A \# B \leq \frac{A + B}{2}$.
\end{itemize}

Many interesting properties of the Wasserstein mean including the log-majorization \cite{BJL18}, order inequalities and Lie-Trotter product formula \cite{HK}, and relationships with other matrix means \cite{LK19} have been found. By using the strict concavity of the function $\log \det: \mathbb{P}_{m} \to \mathbb{R}$, we do not find only the determinantal inequality of the Wasserstein mean in Section 2, but also the equivalent condition that the determinantal equality holds.

The (strictly) positive linear map with its related properties is a very crucial tool to study operator algebra and quantum information theory. Differently from the usual matrix multiplication, tensor (Kronecker) product and Hadamard product are commonly used in matrix equation, image processing, and machine learning due to their algebraic characterizations. The positive linear map also plays an important role to connect between the tensor product and Hadamard product. Applying bounds of the Wasserstein mean verified in \cite{HK} we find in Section 3 inequalities of the Wasserstein mean related with the strictly positive linear map. We finally see in Section 4 the identity and inequalities of the Wasserstein mean with the tensor product and Hadamard product.

\section{Wasserstein mean}

Let $M_{m,k}$ be the set of all $m \times k$ matrices with complex entries. We simply denote as $M_{m} := M_{m, m}$. Let $\mathbb{H}_{m} \subset M_{m}$ be the real vector space of all Hermitian matrices. Let $\mathbb{P}_{m} \subset \mathbb{H}_{m}$ be the open convex cone of all positive definite matrices. For any $A, B \in \mathbb{H}_{m}$ we write $A \leq B$ if $B - A$ is positive semi-definite, and $A < B$ if $B - A$ is positive definite. This is indeed a partial order on $\mathbb{H}_{m}$, known as the Loewner order.

Let $\mathcal{P}(\mathbb{R}^{n})$ be the set of all Borel probability measures on the $n$-dimensional Euclidean space $\mathbb{R}^{n}$. For $1 \leq r < \infty$
\begin{displaymath}
\mathcal{P}^{r}(\mathbb{R}^{n}) = \left\{ \mu \in \mathcal{P}(\mathbb{R}^{n}) : \int_{\mathbb{R}^{n}} \Vert x-y \Vert^{r} \,d\mu(x) < \infty \ \textrm{for any} \ y \in \mathbb{R}^{n} \right\}
\end{displaymath}
Let $\mathcal{P}^{0}(\mathbb{R}^{n})$ be a set of all uniformly distributed probability measures, and let $\mathcal{P}^{\infty}(\mathbb{R}^{n})$ be a set of all probability measures whose support is bounded.

Given $\mu, \nu \in \mathcal{P}^{2}(\mathbb{R}^{n})$ the $2$-Wasserstein distance is defined as
\begin{displaymath}
W_{2}(\mu, \nu) := \left\{ \underset{\pi \in \Pi (\mu, \nu)}{\inf} \int_{\mathbb{R}^{n}} \Vert x-y \Vert^{2} d\pi(x,y) \right\}^{1/2},
\end{displaymath}
where $\Pi (\mu, \nu)$ denotes the set of all couplings on $\mathbb{R}^{n} \times \mathbb{R}^{n}$ with marginals $\mu$ and $\nu$. Especially, the $2$-Wasserstein distance for two Gaussian probabilities $\mu = P(m_{1}, A), \nu = P(m_{2}, B)$ with means $m_{1}, m_{2}$ and covariance matrices $A, B \in \mathbb{P}_{m}$ is given by
\begin{displaymath}
W_{2}^{2}(\mu, \nu) = | m_{1} - m_{2} |^{2} + \tr \left[ A + B - 2 (A^{1/2} B A^{1/2})^{1/2} \right].
\end{displaymath}
Here, we consider the $2$-Wasserstein distance for two Gaussian probabilities with mean $0$ such as
\begin{equation} \label{E:Wasserstein metric}
d(A, B) := \frac{1}{\sqrt{2}} W_{2}(P(0, A), P(0, B)) = \left[ \tr \left( \frac{A + B}{2} \right) - \tr (A^{1/2} B A^{1/2})^{1/2} \right]^{1/2}.
\end{equation}
See \cite{ABCM} for more details.

The $2$-Wasserstein distance \eqref{E:Wasserstein metric} and the unique geodesic for this metric on the open convex cone $\mathbb{P}_{m}$ of positive definite matrices have been recently introduced in \cite{BJL}. This metric is the matrix version of the Hellinger distance
$$ \displaystyle d(\overrightarrow{p}, \overrightarrow{q}) = \left[ \frac{1}{2} \sum_{i=1}^{n} ( \sqrt{p_{i}} - \sqrt{q_{i}} )^{2} \right]^{1/2} $$
for two probability distributions $\overrightarrow{p} = (p_{1}, \dots, p_{n})$ and $\overrightarrow{q} = (q_{1}, \dots, q_{n})$. Moreover, it coincides with the Bures distance of density matrices in quantum information theory and the Wasserstein metric in statistics and the theory of optimal transport. The Bures-Wasserstein metric is a Riemannian metric induced by the inner product
\begin{displaymath}
\langle X, Y \rangle_{A} = \sum_{i,j=1}^{m} \frac{\alpha_{i} \textrm{Re} (\overline{x_{ji}} y_{ji})}{(\alpha_{i} + \alpha_{j})^{2}}
\end{displaymath}
for any $X = [ x_{ij} ]$ and $Y = [ y_{ij} ]$ on the tangent space $T_{A} \mathbb{P}_{m} \equiv \mathbb{H}_{m}$ for each $A \in \mathbb{P}_{m}$, where $\alpha_{1}, \dots, \alpha_{m}$ are positive eigenvalues of $A \in \mathbb{P}_{m}$. The unique geodesic connecting from $A$ to $B$ for the Bures-Wasserstein distance is given by
\begin{displaymath}
A \diamond_{t} B := (1-t)^{2} A + t^{2} B + t(1-t) \left[ (A B)^{1/2} + (B A)^{1/2} \right], \ t \in [0,1].
\end{displaymath}

Let $\mathbb{A} = (A_{1}, \dots, A_{n}) \in \mathbb{P}_{m}^{n}$, and let $\omega = (w_{1}, \dots, w_{n}) \in \Delta_{n}$, the simplex of all positive probability vectors in $\mathbb{R}^{n}$. We consider the following minimization problem
\begin{equation} \label{E:minimization}
\underset{X \in \mathbb{P}_{m}}{\arg \min} \sum_{j=1}^{n} w_{j} d^{2}(X, A_{j}),
\end{equation}
where $d$ is the Bures-Wasserstein distance on $\mathbb{P}_{m}$.
By using tools from non-smooth analysis, convex duality, and the optimal transport theory, it has been proved in Theorem 6.1, \cite{AC} that the above minimization problem has a unique solution in $\mathbb{P}_{m}$. On the other hand, it has been shown in \cite{BJL} that the objective function $\displaystyle f(X) = \sum_{j=1}^{n} w_{j} d^{2}(X, A_{j})$ is strictly convex on $\mathbb{P}_{m}$, by applying the strict concavity of the map $h: \mathbb{P}_{m} \to \mathbb{R}, \ h(X) = \tr (X^{1/2})$. Therefore, we define such a unique minimizer of \eqref{E:minimization} as the \emph{Wasserstein mean}, denoted by $\Omega(\omega; \mathbb{A})$. That is,
\begin{equation} \label{E:Wasserstein mean}
\Omega(\omega; \mathbb{A}) = \underset{X \in \mathbb{P}_{m}}{\arg \min} \sum_{j=1}^{n} w_{j} d^{2}(X, A_{j}).
\end{equation}

To find the unique minimizer of objective function $f: \mathbb{P}_{m} \to \mathbb{R}$, we evaluate the derivative $D f(X)$ and set it equal to zero.
By using matrix differential calculus, we have the following.
\begin{theorem} \cite[Theorem 8]{BJL} \label{T:WassEq}
The Wasserstein mean $\Omega(\omega; \mathbb{A})$ is a unique solution $X \in \mathbb{P}_{m}$ of the nonlinear matrix equation
\begin{displaymath}
I = \sum_{j=1}^{n} w_{j} (A_{j} \# X^{-1}),
\end{displaymath}
equivalently,
\begin{displaymath}
X = \sum_{j=1}^{n} w_{j} (X^{1/2} A_{j} X^{1/2})^{1/2}.
\end{displaymath}
\end{theorem}

\begin{remark} \label{R:comm}
If $A_{i}$'s commute, then they are simultaneously unitarily diagonalizable by Theorem 1.3.21 in \cite{HJ}: there exists a unitary matrix $U$ such that $U A_{i} U^{*}$ are diagonal matrices for all $i$. Then the Wasserstein mean becomes
\begin{displaymath}
\Omega(\omega; \mathbb{A}) = \left[ \sum_{j=1}^{n} w_{j} A_{j}^{1/2} \right]^{2},
\end{displaymath}
which is the $1/2$-power mean of $A_{1}, \dots, A_{n}$ \cite{BJL18}.
\end{remark}

It is known from Theorem 7.6.6 in \cite{HJ} that the map $f: \mathbb{P}_{m} \to \mathbb{R}, \ f(A) = \log \det A$ is strictly concave: for any $A, B \in \mathbb{P}_{m}$ and $t \in [0,1]$
\begin{displaymath}
\log \det ( (1-t) A + t B ) \geq (1-t) \log \det A + t \log \det B,
\end{displaymath}
where equality holds if and only if $A = B$. By induction together with this, we have
\begin{lemma} \label{L:logdet}
Let $A_{1}, \dots, A_{n} \in \mathbb{P}_{m}$, and let $\omega = (w_{1}, \dots, w_{n}) \in \Delta_{n}$. Then
\begin{displaymath}
\log \det \left( \sum_{j=1}^{n} w_{j} A_{j} \right) \geq \sum_{j=1}^{n} w_{j} \log \det A_{j},
\end{displaymath}
where equality holds if and only if $A_{1} = \cdots = A_{n}$.
\end{lemma}

The following shows the determinantal inequality of the Wasserstein mean.

\begin{theorem} \label{T:Wass-det}
Let $\mathbb{A} = (A_{1}, \dots, A_{n}) \in \mathbb{P}_{m}^{n}$, and let $\omega = (w_{1}, \dots, w_{n}) \in \Delta_{n}$. Then
\begin{equation} \label{E:logdet}
\det \Omega(\omega; \mathbb{A}) \geq \prod_{j=1}^{n} (\det A_{j})^{w_{j}},
\end{equation}
where equality holds if and only if $A_{1} = \cdots = A_{n}$.
\end{theorem}

\begin{proof}
Let $X = \Omega(\omega; \mathbb{A})$. Then by Theorem \ref{T:WassEq} $\displaystyle I = \sum_{j=1}^{n} w_{j} (A_{j} \# X^{-1})$, and by Lemma \ref{L:logdet}
\begin{displaymath}
\begin{split}
0 = \log \det \left[ \sum_{j=1}^{n} w_{j} (A_{j} \# X^{-1}) \right]
& \geq \sum_{j=1}^{n} w_{j} \log \det (A_{j} \# X^{-1}) \\
& = \frac{1}{2} \sum_{j=1}^{n} w_{j} \log \det A_{j} - \frac{1}{2} \log \det X.
\end{split}
\end{displaymath}
The last equality follows from the determinantal identity of two-variable geometric mean in (G6). It implies
\begin{displaymath}
\log \det X \geq \sum_{j=1}^{n} w_{j} \log \det A_{j} = \log \left[ \prod_{j=1}^{n} ( \det A_{j} )^{w_{j}} \right].
\end{displaymath}
Taking the exponential function on both sides and applying the fact that the exponential function from $\mathbb{R}$ to $(0, \infty)$ is monotone increasing, we obtain the desired inequality.

Moreover, the equality of \eqref{E:logdet} holds if and only if $A_{i} \# X^{-1} = A_{j} \# X^{-1}$ for all $i$ and $j$. By the definition of geometric mean it is equivalent to $A_{i} = A_{j}$ for all $i$ and $j$.
\end{proof}

\begin{remark}
The Cartan mean $\Lambda(\omega; \mathbb{A})$ is the least squares mean in $\mathbb{P}_{m}$ with respect to the Riemannian trace metric $d_{R} (A, B) = \Vert \log A^{-1/2} B A^{-1/2} \Vert_{2}$:
\begin{displaymath}
\Lambda(\omega; \mathbb{A}) = \underset{X \in \mathbb{P}_{m}}{\arg \min} \sum_{j=1}^{n} w_{j} d_{R}^{2}(X, A_{j}).
\end{displaymath}
By using the $k$-th antisymmetric tensor powers, it has been shown in \cite[Theorem 1]{BJL18} the weak log-majorization between the Wasserstein mean and Cartan mean:
\begin{displaymath}
\lambda(\Lambda(\omega; \mathbb{A})) \prec_{w \log} \lambda(\Omega(\omega; \mathbb{A})),
\end{displaymath}
where $\lambda(A)$ stands for the $m$-tuple of eigenvalues of $A \in \mathbb{P}_{m}$. This is much stronger than our result in Theorem \ref{T:Wass-det}. We do not only provide a different proof, but also provide a sufficient and necessary condition for the determinantal equality by using the concavity of the map $f: \mathbb{P}_{m} \to \mathbb{R}, \ f(A) = \log \det A$.
\end{remark}

\section{Inequalities of the Wasserstein mean}

In \cite{BJL} the arithmetic-Wasserstein means inequality has been shown:
\begin{displaymath}
\Omega(\omega; \mathbb{A}) \leq \sum_{j=1}^{n} w_{j} A_{j} =: \mathcal{A}(\omega; \mathbb{A}).
\end{displaymath}
On the other hand, the Wasserstein-harmonic means inequality does not hold, but a new lower bound of the Wasserstein mean with respect to the Loewner order is found.
\begin{theorem} \cite{HK} \label{T:bounds}
The Wasserstein mean $\Omega(\omega; \mathbb{A})$ satisfies the following inequalities:
\begin{displaymath}
2I - \sum_{j=1}^{n} w_{j} A_{j}^{-1} \leq \Omega(\omega; \mathbb{A}) \leq \sum_{j=1}^{n} w_{j} A_{j}.
\end{displaymath}
\end{theorem}

We call that a linear map $\Phi: M_{m} \to M_{k}$ is positive if $\Phi(A) \geq O$ whenever $A \geq O$, and strictly positive if $\Phi(A) > O$ whenever $A > O$. The map $\Phi$ is said to be unital if $\Phi(I) = I$, where $I$ is the identity matrix. The positive linear map including its related properties is an important tool in operator algebra and quantum information theory. See \cite{Bh} and its bibliographies. We obtain the following inequalities of Wasserstein mean related with the strictly positive and unital linear map.

\begin{lemma} \cite[Theorem 4.4.5]{Bh} \label{L:Phi}
Let $\Phi$ be a positive linear map. Then for any $A,B \in \mathbb{P}_{m}$
\begin{displaymath}
\Phi(A \# B) \leq \Phi(A) \# \Phi(B).
\end{displaymath}
\end{lemma}

\begin{theorem} \label{T:Phi-Wass}
Let $\Phi$ be a strictly positive and unital linear map. Then
\begin{displaymath}
\Phi(\Omega(\omega; \mathbb{A})) \geq 2I - \sum_{j=1}^{n} w_{j} \Phi(A_{j}^{-1}).
\end{displaymath}
Moreover,
\begin{displaymath}
\Phi(\Omega(\omega; \mathbb{A})^{-1}) \geq 2I - \sum_{j=1}^{n} w_{j} \Phi(A_{j}).
\end{displaymath}
\end{theorem}

\begin{proof}
By Theorem \ref{T:bounds} and the positive unital linear map $\Phi$,
\begin{displaymath}
\Phi(\Omega(\omega; \mathbb{A})) \geq \Phi \left( 2I- \sum_{j=1}^{n} w_{j} A_{j}^{-1} \right) \geq  2I - \sum_{j=1}^{n} w_{j} \Phi(A_{j}^{-1}).
\end{displaymath}
So we obtain the first inequality.

To prove the second inequality, let $X = \Omega(\omega; \mathbb{A})$. Then by Theorem \ref{T:WassEq} and the strict positive unital linear map $\Phi$,
\begin{displaymath}
\begin{split}
I & = \Phi(I) = \Phi \left( \sum_{j=1}^{n} w_{j} (A_{j} \# X^{-1}) \right) = \sum_{j=1}^{n} w_{j} \Phi(A_{j} \# X^{-1}) \\
& \leq \sum_{j=1}^{n} w_{j} \Phi(A_{j}) \# \Phi(X^{-1}) \leq \frac{1}{2} \sum_{j=1}^{n} w_{j} \Phi(A_{j}) + \frac{1}{2} \Phi(X^{-1}).
\end{split}
\end{displaymath}
The first inequality follows from Lemma \ref{L:Phi}, and the second inequality follows from the arithmetic-geometric mean inequality in (G7). Solving the above for $\Phi(X^{-1})$ yields we obtain the desired inequality.
\end{proof}

\begin{remark}
By the first inequality in Theorem \ref{T:Phi-Wass} one can easily have
\begin{displaymath}
\Phi(\Omega(\omega; \mathbb{A}^{-1})) \geq 2I - \sum_{j=1}^{n} w_{j} \Phi(A_{j}),
\end{displaymath}
where $\mathbb{A}^{-1} := (A_{1}^{-1}, \dots, A_{n}^{-1}) \in \mathbb{P}_{m}^{n}$. On the other hand, it does not satisfy the self-duality of the Wasserstein mean: $\Omega(\omega; \mathbb{A}^{-1}) \neq \Omega(\omega; \mathbb{A})^{-1}$. It means that the second inequality in Theorem \ref{T:Phi-Wass} could not be derived from the first inequality in Theorem \ref{T:Phi-Wass}.
\end{remark}

\begin{remark}
As an extension of the result in Lemma \ref{L:Phi} the following has been shown in \cite[Corollary 4.5]{LP}:
\begin{displaymath}
\Phi(\Lambda(\omega; \mathbb{A})) \leq \Lambda(\omega; \Phi(A_{1}), \dots, \Phi(A_{n}))
\end{displaymath}
for any positive unital linear map $\Phi$, and the equality holds for any strictly positive unital linear map $\Phi$. Theorem \ref{T:Phi-Wass} tells us the relation between $\Phi(\Omega(\omega; \mathbb{A}))$ and the arithmetic mean of $\Phi(A_{1}), \dots, \Phi(A_{n})$. On the other hand , the order relation between $\Phi(\Omega(\omega; \mathbb{A}))$ and $\Omega(\omega; \Phi(A_{1}), \dots, \Phi(A_{n}))$ is unknown yet.
\end{remark}

\section{Tensor product and Hadamard product}

The \emph{tensor product} $A \otimes B$ of $A=[a_{ij}] \in M_{m,k}$ and $B=[b_{ij}] \in M_{s,t}$ is the $ms \times kt$ matrix:
\begin{displaymath}
A \otimes B :=
\left[
  \begin{array}{ccc}
    a_{11}B & \cdots & a_{1k}B \\
    \vdots & \ddots & \vdots \\
    a_{m1}B & \cdots & a_{mk}B \\
  \end{array}
\right].
\end{displaymath}
One can see easily that the tensor product is bilinear and associative, but not commutative. In addition, the tensor product of two positive definite (positive semidefinite) matrices is positive definite (positive semidefinite, respectively).
We enumerate a few properties of the tensor product that we will use in the following.

\begin{lemma} \label{L:Tensor} \cite[Section 4.3]{Zh}
The tensor product satisfies the following.
\begin{itemize}
\item[(1)] For $A \in M_{m,k}, B \in M_{r,s}, C \in M_{k,l}$ and $D \in M_{s,t}$
    $$ (A \otimes B)(C \otimes D) = AC \otimes BD.$$
\item[(2)] For positive definite matrices $A, B$ and any real number $t$
    $$(A \otimes B)^{t} = A^{t} \otimes B^{t}.$$
\end{itemize}
\end{lemma}

We get the following identity of Wasserstein means related with the tensor product.

\begin{theorem} \label{T:Tensor}
Let $\mathbb{A} = (A_{1}, \dots, A_{n}), \mathbb{B} = (B_{1}, \dots, B_{n}) \in \mathbb{P}_{m}^{n}$, and let $\omega = (w_{1}, \dots, w_{n}), \mu = (\mu_{1}, \dots, \mu_{n}) \in \Delta_{n}$. Then
\begin{displaymath}
\Omega(\omega; \mathbb{A}) \otimes \Omega(\mu; \mathbb{B}) = \Omega(\omega \otimes \mu; \underbrace{A_{1} \otimes B_{1}, \dots, A_{1}\otimes B_{n}}, \dots, \underbrace{A_{n} \otimes B_{1}, \dots, A_{n} \otimes B_{n}})
\end{displaymath}
where $\omega \otimes \mu := (\underbrace{w_{1} \mu_{1}, \dots, w_{1} \mu_{n}}, \dots, \underbrace{w_{n} \mu_{1}, \dots, w_{n} \mu_{n}}) \in \Delta_{n^{2}}.$
\end{theorem}

\begin{proof}
Let $X = \Omega(\omega; \mathbb{A})$ and $Y = \Omega(\mu; \mathbb{B}).$
Applying Theorem \ref{T:WassEq}, the linearity of tensor product, and Lemma \ref{L:Tensor}, we have
\begin{align*}
X \otimes Y & = \left( \sum_{i=1}^{n} w_{i} (X^{1/2} A_{i} X^{1/2})^{1/2} \right) \otimes \left( \sum_{j=1}^{n} \mu_{j} (Y^{1/2} B_{j} Y^{1/2})^{1/2} \right) \\
& = \sum_{i,j=1}^{n} \omega_{i} \mu_{j}((X \otimes Y)^{1/2}(A_{i} \otimes B_{j})(X \otimes Y)^{1/2})^{1/2}.
\end{align*}
Note that $\omega \otimes \mu \in \Delta_{n^{2}}$, and hence, we obtain by Theorem \ref{T:WassEq} that
\begin{displaymath}
X \otimes Y = \Omega(\omega \otimes \mu; A_{1} \otimes B_{1}, \dots, A_{1} \otimes B_{n}, \dots, A_{n} \otimes B_{1}, \dots, A_{n} \otimes B_{n}).
\end{displaymath}
\end{proof}

By the arithmetic-Wasserstein mean inequality in Theorem \ref{T:bounds}, we easily obtain the following.
\begin{corollary} \label{C:Tensor}
Let $\mathbb{A} = (A_{1}, \dots, A_{n}), \mathbb{B} = (B_{1}, \dots, B_{n}) \in \mathbb{P}_{m}^{n}$, and let $\omega = (w_{1}, \dots, w_{n}), \mu = (\mu_{1}, \dots, \mu_{n}) \in \Delta_{n}$. Then
\begin{displaymath}
\Omega(\omega; \mathbb{A}) \otimes \Omega(\mu; \mathbb{B}) \leq \mathcal{A}(\omega \otimes \mu; \underbrace{A_{1} \otimes B_{1}, \dots, A_{1}\otimes B_{n}}, \dots, \underbrace{A_{n} \otimes B_{1}, \dots, A_{n} \otimes B_{n}}),
\end{displaymath}
where $\displaystyle \mathcal{A}(\omega \otimes \mu; A_{1} \otimes B_{1}, \dots, A_{1} \otimes B_{n}, \dots, A_{n} \otimes B_{1}, \dots, A_{n} \otimes B_{n}) = \sum_{i,j=1}^{n} w_{i} \mu_{j} A_{i} \otimes B_{j}$.
\end{corollary}

The \emph{Hadamard product} (or the \emph{Schur product}) $A \circ B$ of $A = [a_{ij}]$ and $B = [b_{ij}]$ in $M_{m, k}$ is the $m \times k$ matrix:
\begin{displaymath}
\displaystyle A \circ B := [a_{ij} b_{ij}].
\end{displaymath}
Simply one can see that the Hadamard product is the entry-wise product and gives us a binary operation on $M_{m, k}$. Moreover, the Hadamard product is bilinear, commutative, and associative. Moreover, the Hadamard product preserves positivity; the Hadamard product of two positive definite (positive semidefinite) matrices is again positive definite (positive semidefinite, respectively). This is known as the Schur product theorem.

We show the inequality of Wasserstein means related with the Hadamard product.
\begin{lemma}\cite[Lemma 4]{An} \label{L:An}
There exists a strictly positive and unital linear map $\Phi$ such that for any $A, B \in M_{m}$
\begin{displaymath}
\Phi(A \otimes B)= A \circ B.
\end{displaymath}
\end{lemma}

\begin{theorem} \label{T:Hada}
Let $\mathbb{A} = (A_{1}, \dots, A_{n}), \mathbb{B} = (B_{1}, \dots, B_{n}) \in \mathbb{P}_{m}^{n}$ and let $\omega = (w_{1}, \dots, w_{n}), \mu = (\mu_{1}, \dots, \mu_{n}) \in \Delta_{n}$. Then
\begin{displaymath}
\Omega(\omega; \mathbb{A}) \circ \Omega(\mu;\mathbb{B}) \leq \mathcal{A}(\omega \otimes \mu;\underbrace{A_{1} \circ B_{1},\dots, A_{1} \circ B_{n}},\dots, \underbrace{A_{n} \circ B_{1}, \dots, A_{n} \circ B_{n}}).
\end{displaymath}
\end{theorem}

\begin{proof}
Using Corollary \ref{C:Tensor} and the strictly positive linear map $\Phi$ in Lemma \ref{L:An}, we get
\begin{align*}
& \Omega(\omega; \mathbb{A}) \circ \Omega(\mu;\mathbb{B}) \\
& = \Phi(\Omega(\omega; \mathbb{A}) \otimes \Omega(\mu;\mathbb{B})) \\
& \leq \Phi ( \mathcal{A}(\omega \otimes \mu; A_{1} \otimes B_{1}, \dots, A_{1}\otimes B_{n}, \dots, A_{n} \otimes B_{1}, \dots, A_{n} \otimes B_{n} )) \\
& = \mathcal{A}(\omega \otimes \mu; \Phi(A_{1} \otimes B_{1}),\dots,\Phi( A_{1} \otimes B_{n}),\dots, \Phi(A_{n} \otimes B_{1}), \dots, \Phi(A_{n} \otimes B_{n}) )\\
& = \mathcal{A}(\omega \otimes \mu; A_{1} \circ B_{1},\dots, A_{1} \circ B_{n},\dots, A_{n} \circ B_{1}, \dots, A_{n} \circ B_{n}).
\end{align*}
\end{proof}

\begin{proposition}
Let $A, B, C, D \in \mathbb{P}_{m}$ such that $A B = B A$ and $C D = D C$. Then
\begin{displaymath}
(A B + B A) \circ (C D + D C) - (A^{2} + B^{2}) \circ (C^{2} + D^{2}) \leq \frac{1}{2} (A - B)^{2} \circ (C - D)^{2}.
\end{displaymath}
\end{proposition}

\begin{proof}
Since $A$ and $B$ commute, so do $A^{2}$ and $B^{2}$. Moreover, $C^{2}$ and $D^{2}$ commute. By Theorem \ref{T:Hada} together with Remark \ref{R:comm} for $\omega = \mu = (1/2, 1/2)$
\begin{displaymath}
\begin{split}
& \Omega (1/2, 1/2; A^{2}, B^{2}) \circ \Omega (1/2, 1/2; C^{2}, D^{2}) = \left( \frac{A + B}{2} \right)^{2} \circ \left( \frac{C + D}{2} \right)^{2} \\
& \leq \frac{1}{4} (A^{2} \circ C^{2} + A^{2} \circ D^{2} + B^{2} \circ C^{2} + B^{2} \circ D^{2}) = \frac{1}{4} (A^{2} + B^{2}) \circ (C^{2} + D^{2}).
\end{split}
\end{displaymath}
It reduces to
\begin{displaymath}
(A + B)^{2} \circ (C D + D C) - (A - B)^{2} \circ (C^{2} + D^{2}) \leq 2 (A^{2} + B^{2}) \circ (C^{2} + D^{2}).
\end{displaymath}
Since the left-hand side is equivalent to $2 (A B + B A) \circ (C D + D C) - (A - B)^{2} \circ (C - D)^{2}$, we obtain the desired inequality by simplification.
\end{proof}

We show another inequality of Wasserstein means related with the Hadamard product.
\begin{lemma} \cite[Section 7.7]{Zh} \label{L:Hada}
For $A, B \in \mathbb{P}_{m}$
\begin{displaymath}
(A \circ B)^{-1} \leq A^{-1} \circ B^{-1} \leq \frac{(\lambda_{1} + \lambda_{m})^{2}}{4 \lambda_{1} \lambda_{m}} (A \circ B)^{-1},
\end{displaymath}
where $\lambda_{1}$ and $\lambda_{m}$ are the largest and smallest eigenvalues of $A \otimes B$, respectively.
\end{lemma}

\begin{remark} \label{R:Kantorovich}
For $0 < p \leq q$ the value $\displaystyle \frac{(p + q)^{2}}{4 p q}$ is known as the Kantorovich constant. One can rewrite it as $\displaystyle f(r) = \frac{(r + 1)^{2}}{4 r}$ for $r = q / p \geq 1$, and $f$ is increasing on $r \geq 1$. It has been widely used in the converse inequalities of the weighted arithmetic, geometric, and harmonic means \cite{FFNPS, KL}.
\end{remark}

\begin{proposition} \label{P:Hada}
Let $\mathbb{A} = (A_{1}, \dots, A_{n}), \mathbb{B} = (B_{1}, \dots, B_{n}) \in \mathbb{P}_{m}^{n}$. Assume that $\alpha_{i} I \leq A_{i} \leq \beta_{i} I$ and $\gamma_{i} I \leq B_{i} \leq \delta_{i} I$ for all $i = 1, \dots, n$, where $\alpha_{i}, \beta_{i}, \gamma_{i}, \delta_{i} > 0$. Let $\omega = (w_{1}, \dots, w_{n}), \mu = (\mu_{1}, \dots, \mu_{n}) \in \Delta_{n}$. Let $X = \Omega(\omega; \mathbb{A})$ and $Y = \Omega(\mu; \mathbb{B})$. Then
\begin{displaymath}
X \circ Y \leq \frac{\alpha \gamma + \beta \delta}{2 \sqrt{\alpha \beta \gamma \delta}} \sum_{i,j=1}^{n} w_{i} \mu_{j} \left[ (X \circ Y)^{1/2} (A_{i} \circ B_{j}) (X \circ Y)^{1/2} \right]^{1/2},
\end{displaymath}
where $\alpha := \underset{1 \leq i \leq n}{\min} \{ \alpha_{i} \}, \beta := \underset{1 \leq i \leq n}{\max} \{ \beta_{i} \}, \gamma := \underset{1 \leq i \leq n}{\min} \{ \gamma_{i} \}$, and $\delta := \underset{1 \leq i \leq n}{\max} \{ \delta_{i} \}$.
\end{proposition}

\begin{proof}
Let $X = \Omega(\omega; \mathbb{A})$ and $Y = \Omega(\mu; \mathbb{B})$. Then $\displaystyle I = \sum_{i=1}^{n} w_{j} (X^{-1} \# A_{i})$ and $\displaystyle I = \sum_{j=1}^{n} \mu_{j} (Y^{-1} \# B_{j})$ by Theorem \ref{T:WassEq}. So
\begin{displaymath}
\begin{split}
I = I \circ I & = \sum_{i,j=1}^{n} w_{i} \mu_{j} (X^{-1} \# A_{i}) \circ (Y^{-1} \# B_{j}) \\
& \leq \sum_{i,j=1}^{n} w_{i} \mu_{j} (X^{-1} \circ Y^{-1}) \# (A_{i} \circ B_{j}) \\
& \leq \frac{\alpha \gamma + \beta \delta}{2 \sqrt{\alpha \beta \gamma \delta}} \sum_{i,j=1}^{n} w_{i} \mu_{j} (X \circ Y)^{-1} \# (A_{i} \circ B_{j}).
\end{split}
\end{displaymath}
The second equality follows from the linearity of Hadamard product, and the first inequality follows from Lemma 3.1 in \cite{LK}.

We verify more details for the second inequality. Indeed, $\alpha_{i} I \leq A_{i} \leq \beta_{i} I$ implies $\alpha I \leq A_{i} \leq \beta I$, so $\alpha I \leq X \leq \beta I$ by Lemma 2.4 in \cite{LK19}. Similarly, we have $\gamma I \leq Y \leq \delta I$, and thus, $\alpha \gamma I \leq X \otimes Y \leq \beta \delta I$. So by Lemma \ref{L:Hada} together with Remark \ref{R:Kantorovich}, the monotonicity of geometric mean in (G3), and the joint homogeneity of geometric mean in (G1), we have
\begin{displaymath}
\begin{split}
(X^{-1} \circ Y^{-1}) \# (A_{i} \circ B_{j}) & \leq \left[ \frac{(\alpha \gamma + \beta \delta)^{2}}{4 \alpha \beta \gamma \delta} (X \circ Y)^{-1} \right] \# (A_{i} \circ B_{j}) \\
& = \frac{\alpha \gamma + \beta \delta}{2 \sqrt{\alpha \beta \gamma \delta}} \left[ (X \circ Y)^{-1} \# (A_{i} \circ B_{j}) \right].
\end{split}
\end{displaymath}
Taking the congruence transformation by $(X \circ Y)^{1/2}$ in the above, we obtain the desired inequality.
\end{proof}

\begin{remark}
Note in Proposition \ref{P:Hada} that $\alpha_{i}$ and $\beta_{i}$ can be taken as the smallest and largest eigenvalues of $A_{i}$, and $\gamma_{i}, \delta_{i}$ as the smallest and largest eigenvalues of $B_{i}$ for $i = 1, \dots, n$. If we assume that $\alpha I \leq A_{i}, B_{i} \leq \beta I$ for all $i$, then $X = \Omega(\omega; \mathbb{A})$ and $Y = \Omega(\mu; \mathbb{B})$ satisfy
\begin{displaymath}
X \circ Y \leq \frac{1}{2} \left( \frac{\beta}{\alpha} + \frac{\alpha}{\beta} \right) \sum_{i,j=1}^{n} w_{i} \mu_{j} \left[ (X \circ Y)^{1/2} (A_{i} \circ B_{j}) (X \circ Y)^{1/2} \right]^{1/2}.
\end{displaymath}
\end{remark}

By Jensen type inequalities in \cite{HP} we have that for every contraction $X$
\begin{displaymath}
\begin{split}
(X^{*} A X)^{p} & \leq X^{*} A^{p} X \hspace{5mm} \textrm{if} \ \ 1 \leq p \leq 2, \\
(X^{*} A X)^{p} & \geq X^{*} A^{p} X \hspace{5mm} \textrm{if} \ \ 0 \leq p \leq 1.
\end{split}
\end{displaymath}
Applying the above inequalities we obtain in \cite{LK19} that for any invertible matrix $X$ whose inverse $X^{-1}$ is a contraction,
\begin{equation} \label{E:Hansen_modified}
(X^{*} A X)^{p} \leq X^{*} A^{p} X \hspace{5mm} \textrm{if} \ \ 0 \leq p \leq 1.
\end{equation}

\begin{theorem}
Let $X = \Omega(\omega; \mathbb{A})$ and $Y = \Omega(\mu; \mathbb{B})$ as in Proposition \ref{P:Hada}. If $X^{-1}$ and $Y^{-1}$ are contractions, then
\begin{displaymath}
\sum_{i,j=1}^{n} w_{i} \mu_{j} (A_{i} \circ B_{j})^{1/2} \geq \frac{2 \sqrt{\alpha \beta \gamma \delta}}{\alpha \gamma + \beta \delta} I.
\end{displaymath}
\end{theorem}

\begin{proof}
Since $X = \Omega(\omega; \mathbb{A}) \in \mathbb{P}_{m}$ and $Y = \Omega(\mu; \mathbb{B}) \in \mathbb{P}_{m}$, $X^{-1}$ and $Y^{-1}$ are contractions if and only if $X^{-1}, Y^{-1} \leq I$. By Lemma \ref{L:Hada} we have
\begin{displaymath}
(X \circ Y)^{-1} \leq X^{-1} \circ Y^{-1} \leq I \circ I = I.
\end{displaymath}
So $(X \circ Y)^{-1}$ is a contraction, which yields that $(X \circ Y)^{-1/2}$ is also a contraction. Applying \eqref{E:Hansen_modified} to Proposition \ref{P:Hada} implies
\begin{displaymath}
X \circ Y \leq \frac{\alpha \gamma + \beta \delta}{2 \sqrt{\alpha \beta \gamma \delta}} \sum_{i,j=1}^{n} w_{i} \mu_{j} \left[ (X \circ Y)^{1/2} (A_{i} \circ B_{j})^{1/2} (X \circ Y)^{1/2} \right].
\end{displaymath}
Taking the congruence transformation by $(X \circ Y)^{-1/2}$ we get
\begin{displaymath}
I \leq \frac{\alpha \gamma + \beta \delta}{2 \sqrt{\alpha \beta \gamma \delta}} \sum_{i,j=1}^{n} w_{i} \mu_{j} (A_{i} \circ B_{j})^{1/2},
\end{displaymath}
which is equivalent to the desired inequality.
\end{proof}

\vspace{4mm}

\textbf{Acknowledgement} \\

This work was supported by the National Research Foundation of Korea (NRF) grant funded by the Korea government (No. NRF-2018R1C1B6001394).

\end{document}